\documentclass{article}
  \usepackage{arydshln,amsmath,url,listings,color,amssymb,hyperref,amsthm,tikz,breqn}
  \newtheorem{lem}{Lemma}
  \newtheorem{thm}{Theorem}
  \newtheorem{cor}{Corollary}
  \newtheorem{conj}{Conjecture}
  \usetikzlibrary{shapes.misc}
  \usepackage{pgfplots} \pgfplotsset{compat=1.14}
  \usepackage{wrapfig,tikzscale}

  \title{Asymptotic expansions of truncated hypergeometric series for $1/\pi$}
  \author{Lorenz Milla\footnote{Heidelberg, Germany, ORCID 0000-0002-6365-5958},\quad
        Chao-Ping Chen\footnote{School of Mathematics and Information Science, Henan Polytechnic University, Jiaozuo City 454003, Henan Province, China}}
  \date{23/07/2024}

\begin{document}

  \maketitle
  
  \begin{abstract}
    In this paper, we consider rational hypergeometric series of the form
    \[\frac{p}{\pi}= \sum_{k=0}^\infty u_k\quad\text{with}\quad u_k=\frac{\left(\frac{1}{2}\right)_k \left(q\right)_k \left(1-q\right)_k}{(k!)^3}(r+s\,k)\,t^k,\]
    where $(a)_k$ denotes the Pochhammer symbol and $p,q,r,s,t$ are algebraic coefficients.
    Using only the first $n+1$ terms of this series, we define the remainder
    \[\mathcal{R}_n = \frac{p}{\pi} - \sum_{k=0}^n u_k=\sum_{k=n+1}^\infty u_k.\]
     We consider an asymptotic expansion of $\mathcal{R}_n$. More precisely, we provide a recursive relation for determining the coefficients $c_j$ such that
    \[  \mathcal{R}_n    = \frac{\left(\frac{1}{2}\right)_n \left(q\right)_n \left(1-q\right)_n}{n!^3}nt^n\left(\sum_{j=0}^{J-1}\frac{c_j}{n^j}+\mathcal{O}\left(n^{-J}\right)\right),\qquad n \rightarrow \infty.\]
    Here we need $J<\infty$ to approximate $\mathcal{R}_n$, because (like the Stirling series) this series diverges if $J\rightarrow\infty$.
    By applying our recursive relation to the Chudnovsky formula, we solve an open problem posed by Han and Chen.
  \end{abstract}


\section{Introduction}
Ramanujan \cite{rama1914} gave 17 series for $1/\pi$, the fastest among them is
\begin{align}
    \frac{1}{2\pi\sqrt{2}}&=\sum_{k=0}^{\infty}\frac{1}{2^{8k}}\frac{(4k)!}{(k!)^4}\,\frac{1103+26390\,k}{99^{4k+2}}.
\end{align}
Han and Chen  \cite{HanChen} defined the remainder
\begin{align}
        R_n &= \frac{1}{2\pi\sqrt{2}} - \sum_{k=0}^n \frac{1}{2^{8k}}\frac{(4k)!}{(k!)^4}\,\frac{1103+26390\,k}{99^{4k+2}}
\end{align}
and provided a recursive relation for determining the coefficients $r_j$ in the asymptotic expansion
\begin{align}
        R_n &\sim \frac{1}{2^{8n}}\frac{(4n)!}{(n!)^4}\,\frac{n}{99^{4n+2}}\sum_{j=0}^\infty \frac{r_j}{n^j},\qquad n \rightarrow \infty.
\end{align}

The Chudnovsky brothers \cite{chud1988} found an even faster series, which was used in most recent computations of $\pi$ -- it reads
\begin{align}
        \frac{3}{2\pi\sqrt{10005}} &= \sum_{k=0}^\infty(-1)^k \frac{(6k)!}{12^{3k}(3k)!(k!)^3} \frac{13591409 + 545140134\,k}{ 53360^{3k+2}}.
\end{align}
Berndt and Chan gave more details on this series in \cite{Berndt}.
Milla provided a detailed proof of this series using elliptic curves and the Picard Fuchs differential equation in \cite{MillaArxiv} and calculated the coefficients in the Chudnovskys' series using the integrality of certain non-holomorphic modular functions in \cite{millaRamJ}.

Cohen and Guillera \cite{cohenGuillera} summarized all known rational hypergeometric series for $1/\pi$. They can be expressed in the following form with Pochhammer symbols $(a)_k$:
\begin{align}
    \frac{p}{\pi}&= \sum_{k=0}^\infty \frac{\left(\frac{1}{2}\right)_k \left(q\right)_k \left(1-q\right)_k}{(k!)^3}(r+s\,k)\,t^k, \qquad (a)_k=\frac{\Gamma(a+k)}{\Gamma(k)},\label{typersz}
\end{align}
where the coefficients $p,q,r,s,t$ are algebraic numbers. The possible combinations are listed in Table \ref{L0N1}.
\begin{table}[ht!]
    \centering
\begin{tabular}{c|c|c|c|c|c}
\textbf{\#} & \textbf{p} & \textbf{q} & \textbf{r} & \textbf{s} & \textbf{t} \\
\hline
\textbf{1} & $5\sqrt{15}$ & $1/6$ & $8$ & $63$ & $-64/125$ \\
\textbf{2} & $32\sqrt{2}$ & $1/6$ & $15$ & $154$ & $-27/512$ \\
\textbf{3} & $32\sqrt{6}$ & $1/6$ & $25$ & $342$ & $-1/512$ \\
\textbf{4} & $160\sqrt{30}/9$ & $1/6$ & $31$ & $506$ & $-9/64000$ \\
\textbf{5} & $640\sqrt{15}/3$ & $1/6$ & $263$ & $5418$ & $-1/512000$ \\
\textbf{6} & $1760\sqrt{330}$ & $1/6$ & $10177$ & $261702$ & $-1/440^3$ \\
\textbf{7} & $426880\sqrt{10005}$ & $1/6$ & $13591409$ & $545140134$ & $-1/53360^3$ \\
\textbf{8} & $5\sqrt{5}$ & $1/6$ & $3$ & $28$ & $27/125$ \\
\textbf{9} & $5\sqrt{15}/6$ & $1/6$ & $1$ & $11$ & $4/125$ \\
\textbf{10} & $11\sqrt{33}/4$ & $1/6$ & $5$ & $63$ & $8/1331$ \\
\textbf{11} & $85\sqrt{255}/54$ & $1/6$ & $8$ & $133$ & $64/614125$ \\
\hdashline
\textbf{12} & $8$ & $1/4$ & $3$ & $20$ & $-1/4$ \\
\textbf{13} & $72$ & $1/4$ & $23$ & $260$ & $-1/324$ \\
\textbf{14} & $3528$ & $1/4$ & $1123$ & $21460$ & $-1/777924$ \\
\textbf{15} & $9\sqrt{7}$ & $1/4$ & $8$ & $65$ & $-256/3969$ \\
\textbf{16} & $16\sqrt{3}/3$ & $1/4$ & $3$ & $28$ & $-1/48$ \\
\textbf{17} & $288\sqrt{5}/5$ & $1/4$ & $41$ & $644$ & $-1/25920$ \\
\textbf{18} & $9/2$ & $1/4$ & $1$ & $7$ & $32/81$ \\
\textbf{19} & $2\sqrt{3}$ & $1/4$ & $1$ & $8$ & $1/9$ \\
\textbf{20} & $9\sqrt{2}/4$ & $1/4$ & $1$ & $10$ & $1/81$ \\
\textbf{21} & $49\sqrt{3}/9$ & $1/4$ & $3$ & $40$ & $1/2401$ \\
\textbf{22} & $18\sqrt{11}$ & $1/4$ & $19$ & $280$ & $1/9801$ \\
\textbf{23} & $9801\sqrt{2}/4$ & $1/4$ & $1103$ & $26390$ & $1/99^4$ \\
\hdashline
\textbf{24} & $12\sqrt{3}$ & $1/3$ & $7$ & $51$ & $-1/16$ \\
\textbf{25} & $96\sqrt{3}$ & $1/3$ & $53$ & $615$ & $-1/1024$ \\
\textbf{26} & $1500\sqrt{3}$ & $1/3$ & $827$ & $14151$ & $-1/250000$ \\
\textbf{27} & $4\sqrt{3}/3$ & $1/3$ & $1$ & $5$ & $-9/16$ \\
\textbf{28} & $4\sqrt{15}/5$ & $1/3$ & $1$ & $9$ & $-1/80$ \\
\textbf{29} & $108\sqrt{7}/7$ & $1/3$ & $13$ & $165$ & $-1/3024$ \\
\textbf{30} & $3\sqrt{3}$ & $1/3$ & $1$ & $6$ & $1/2$ \\
\textbf{31} & $27/4$ & $1/3$ & $2$ & $15$ & $2/27$ \\
\textbf{32} & $15\sqrt{3}/2$ & $1/3$ & $4$ & $33$ & $4/125$ \\
\hdashline
\textbf{33} & $2$ & $1/2$ & $1$ & $4$ & $-1$ \\
\textbf{34} & $2\sqrt{2}$ & $1/2$ & $1$ & $6$ & $-1/8$ \\
\textbf{35} & $4$ & $1/2$ & $1$ & $6$ & $1/4$ \\
\textbf{36} & $16$ & $1/2$ & $5$ & $42$ & $1/64$ 

\end{tabular}
\vspace{6pt}
    \caption{Algebraic coefficients $p,q,r,s,t$ in formulae of type (\ref{typersz}).}
    \label{L0N1}
\end{table}

In this paper, we define the remainder
\begin{align}
        \mathcal{R}_n &= \frac{p}{\pi} - \sum_{k=0}^n \frac{\left(\frac{1}{2}\right)_k \left(q\right)_k \left(1-q\right)_k}{(k!)^3}(r+s\,k)\,t^k \label{defiRn}
\end{align}
and prove a new way (Theorem \ref{ThmNew}) to determine the coefficients $c_j$ in the following asymptotic expansion of $\mathcal{R}_n$ at $n\rightarrow\infty$:
\begin{align}
  \mathcal{R}_n
    = \frac{\left(\frac{1}{2}\right)_n \left(q\right)_n \left(1-q\right)_n}{n!^3}nt^n\left(\sum_{j=0}^{J-1}\frac{c_j}{n^j}+\mathcal{O}\left(n^{-J}\right)\right).
\end{align}
Like the Stirling series for $n!$, the series $\sum_{j=0}^{J-1}\frac{c_j}{n^j}$ diverges if $J\rightarrow\infty$. Therefore, it is necessary to select an appropriate finite $J<\infty$ to approximate $\mathcal{R}_n$.

Applying Theorem \ref{ThmNew} to the Chudnovskys' series solves an open problem posed by Han and Chen \cite{HanChen}.
Our proof does not rely on specialized software -- it uses a technique described by Milla in \cite[Sect. 3]{millaRamJ2}.
Throughout this paper, $\mathbb{N}$ represents the set of positive integers and $\mathbb{N}_0:=\mathbb{N}\cup\{0\}$. 


\section{Convergence of asymptotic expansion}
In this section, we prove the existence and convergence of an asymptotic expansion of $\mathcal{R}_n$ at $n\rightarrow\infty$. However, we do not directly compute the coefficients $c_j$ of this expansion yet. Our proof begins with the Stirling series.
\begin{lem}\label{lem1}
    For $0<q<1$ and $n\geq 1$ we define $\alpha_n$ as follows:
    $$\frac{\left(\frac{1}{2}\right)_n \left(q\right)_n \left(1-q\right)_n}{n!^3}
    =  \frac{\sin(\pi q)}{(\pi n)^{3/2}}\exp(\alpha_n).$$
    Then, for $J\in\mathbb N$, we have the following asymptotic expansion of $\alpha_n$ at $n\rightarrow\infty$:
    $$\alpha_n =  \sum_{j=1}^{J-1}\frac{f_j}{n^j}+\mathcal{O}\left(n^{-J}\right)$$
    with coefficients $(f_j)_{j\in\mathbb N}$.
    Two further remarks:\begin{itemize}
    \item A way to compute the coefficients $f_j$ is given in (\ref{koeffan}).
    \item Caution: The series $\sum_{j=1}^\infty\frac{f_j}{n^j}$ diverges for all $n\in\mathbb N$, thus we need $J<\infty$.\end{itemize}
\end{lem}
\begin{proof}

The Stirling series for the $\Gamma$-function is (see \cite{sasvari}):
\begin{align}
    \ln(\Gamma(x+1)) &= \frac{1}{2}\ln(2\pi)+\left(x+\frac{1}{2}\right)\ln(x)-x+\Theta(x)\label{acht}\\
    \text{with }\Theta(x)&= \int_0^\infty \left(\frac{1}{e^t-1}-\frac{1}{t}+\frac{1}{2}\right)e^{-xt}\frac{1}{t}dt
\end{align}
with the following asymptotic expansion of $\Theta(x)$ for $x\rightarrow\infty$, where $B_{2l}$ denote the Bernoulli numbers:
\begin{align}
    \Theta(x) &=
    \sum_{l=1}^L \frac{B_{2l}}{2l(2l-1)x^{2l-1}}+\mathcal{O}\left(x^{-(2L+1)}\right)\qquad\text{if } L<\infty.\label{thetastirling}
\end{align}
This is a divergent series if $L\rightarrow\infty$, but it envelops $\Theta(x)$ if $L<\infty$ (see \cite{sasvari}):
\begin{align}
       \sum_{l=1}^{2L} \frac{B_{2l}}{2l(2l-1)x^{2l-1}}
       < \Theta(x) <
    \sum_{l=1}^{2L+1} \frac{B_{2l}}{2l(2l-1)x^{2l-1}} \qquad\text{if } L<\infty.
\end{align}
From (\ref{acht}) we deduce:
\begin{align}
    \ln\left(\frac{\Gamma(n+1-q)}{\Gamma(n+1)}\right)
    &=\left(n+\frac{1}{2}-q\right)\ln\left(1-\frac{q}{n}\right) +q\nonumber\\
    &~~~~- q\ln(n)+\Theta(n-q)-\Theta(n)\nonumber\\
    &= \sum_{l=1}^\infty\frac{A_{q,l}}{n^l}- q\ln(n)+\Theta(n-q)-\Theta(n)\label{eq1}\\
    \text{with } A_{q,l} &= \left(q-\frac{1}{2}\right)\frac{q^l}{l}-\frac{q^{l+1}}{l+1} = -\frac{(l + 1 - 2 q) q^l}{2 l (l+1)}\nonumber
\end{align}
Now we use $(q)_n={\Gamma(n+q)}/{\Gamma(q)}$ in the definition of $\alpha_n$ from Lemma \ref{lem1}: 
\begin{align}
    \alpha_n &= \ln\left(\frac{\left(\frac{1}{2}\right)_n \left(q\right)_n \left(1-q\right)_n}{n!^3}\cdot\frac{(\pi n)^{3/2}}{\sin(\pi q)}\right)\label{defan}\\
   &= \ln\left(\frac{\Gamma(n+1/2)\Gamma(n+q)\Gamma(n+1-q)}{\Gamma(n+1)^3}\cdot n^{1/2}\cdot n^q\cdot n^{1-q}\right)\label{vierzehn}
\end{align}
In the last step, we used $\Gamma(q)\Gamma(1-q)= {\pi}/{\sin(\pi q)}$ and $\Gamma\left(\frac{1}{2}\right)=\sqrt{\pi}$.
Applying (\ref{eq1}) to (\ref{vierzehn}) yields
\begin{align}
    \alpha_n &= \sum_{l=1}^\infty\frac{A_{q,l}+A_{1-q,l}+A_{1/2,l}}{n^l}\nonumber\\
    &~~~~+\Theta(n-q)+\Theta(n-(1-q))+\Theta(n-1/2)-3\Theta(n).\label{asdf}
\end{align}
To express $\Theta(n-q)$ as an asymptotic series for $n\rightarrow\infty$, we use
\begin{align}
    \frac{1}{(n-q)^{2l-1}}&=\frac{1}{n^{2l-1}}\frac{1}{(1-q/n)^{2l-1}}
    =\frac{1}{n^{2l-1}}\sum_{m=0}^\infty \binom{2l-2+m}{m}\frac{q^m}{n^m}.
\end{align}
This series is convergent for all $n\in\mathbb N$, because $0<q<1$. Next, we have
\begin{align}
    \Theta(n-q) &= \sum_{l=1}^L \frac{B_{2l}}{2l(2l-1)(n-q)^{2l-1}}+\mathcal{O}\left((n-q)^{-(2L+1)}\right)\nonumber\\
    &= \sum_{l=1}^L \sum_{m=0}^\infty \frac{B_{2l}\binom{2l-2+m}{m}q^m}{2l(2l-1)n^{m+2l-1}}+\mathcal{O}\left(n^{-(2L+1)}\right)\nonumber\\
    &= \sum_{l=1}^L \sum_{m=0}^{2L+1-2l} \frac{B_{2l}\binom{2l-2+m}{m}q^m}{2l(2l-1)n^{m+2l-1}}+\mathcal{O}\left(n^{-(2L+1)}\right).\label{bernou}
\end{align}
Substituting (\ref{bernou}) and (\ref{thetastirling}) into (\ref{asdf}) gives the expansion of $\alpha_n$:
\begin{align}
    \alpha_n &= \sum_{l=1}^L \frac{B_{2l}}{2l(2l-1)n^{2l-1}}\left(-3+\sum_{m=0}^{2L+1-2l} \frac{\binom{2l-2+m}{m}\left(q^m+(1-q)^m+\left(\frac{1}{2}\right)^m\right)}{n^{m}}\right)\nonumber\\
    &~~~~+\sum_{l=1}^{2L}\frac{A_{q,l}+A_{1-q,l}+A_{1/2,l}}{n^l}+\mathcal{O}\left(n^{-(2L+1)}\right)\label{koeffan}
\end{align}
This finishes the proof of Lemma \ref{lem1}.
As mentioned in Lemma \ref{lem1}, this expansion for $\alpha_n$ diverges when $J=2L+1\rightarrow\infty$, since the Stirling series involving Bernoulli numbers is divergent, $B_{2 n} \sim (-1)^{n + 1} 4\sqrt {\pi n} \left(\frac{n}{\pi e}\right)^{2 n}$.
\end{proof}
\pagebreak

\begin{lem}\label{lem2}
    Let $\mathcal{R}_n$ denote the remainder of the truncated hypergeometric series (see (\ref{defiRn})). Let $\left|t\right|<1$ and $0<q<1$ and $J\in\mathbb N$. 
    Then $\mathcal{R}_n$ has the following asymptotic expansion at $n\rightarrow\infty$:
    \begin{align*}
    \mathcal{R}_n
    = \frac{\left(\frac{1}{2}\right)_n \left(q\right)_n \left(1-q\right)_n}{n!^3}nt^n\left(\sum_{j=0}^{J-1}\frac{c_j}{n^j}+\mathcal{O}\left(n^{-J}\right)\right),
    \end{align*}
    with coefficients $(c_j)_{j\in\mathbb N_0}$.
    Two further remarks:\begin{itemize}
    \item An efficient way to compute the coefficients $c_j$ is proven in Theorem \ref{ThmNew}.
    \item Caution: The series $\sum_{j=0}^\infty\frac{c_j}{n^j}$ diverges for all $n\in\mathbb N$, thus we need $J<\infty$.\end{itemize}
\end{lem}
\begin{proof}
First, we define $u_n$ and $\mathcal{F}_n$ and use Lemma \ref{lem1}:
\begin{align}
u_n &= \frac{\left(\frac{1}{2}\right)_n \left(q\right)_n \left(1-q\right)_n}{n!^3}(r+sn) t^n = \frac{\sin(\pi q)}{(\pi n)^{3/2}}\exp(\alpha_n)(r+sn) t^n\nonumber\\
\mathcal{F}_n &= \frac{\left(\frac{1}{2}\right)_n \left(q\right)_n \left(1-q\right)_n}{n!^3}n t^n = \frac{\sin(\pi q)}{(\pi n)^{3/2}}\exp(\alpha_n)n t^n\nonumber
\end{align}
This yields the following representation of $\frac{\mathcal{R}_n}{\mathcal{F}_n}$:
\begin{align}
    \frac{\mathcal{R}_n}{\mathcal{F}_n} = \sum_{k=1}^\infty \frac{ u_{n+k}}{\mathcal{F}_n}
    &= \sum_{k=1}^\infty \frac{ \frac{\sin(\pi q)}{(\pi (n+k))^{3/2}}\exp(\alpha_{n+k})\left(r+s(n+k)\right)t^{n+k}}
    {\frac{\sin(\pi q)}{(\pi n)^{3/2}}\exp(\alpha_{n})nt^n}\nonumber\\
    &= \sum_{k=1}^\infty \frac{\left(s+\frac{r+sk}{n}\right)t^k}{\left(1+\frac{k}{n}\right)^{3/2}} \exp(\alpha_{n+k}-\alpha_n).\label{rf}
\end{align}
Here we encounter a problem: we aim for an asymptotic expansion in negative powers of $n$, yet in $\alpha_{n+k}$ we have terms like $(n+k)^{-j}$ and in (\ref{rf}), there is also a factor $(1+k/n)^{-3/2}$. Unfortunately, the known expansion with binomial coefficients
\begin{align}
    \frac{1}{\left(1+\frac{k}{n}\right)^{j}}=\sum_{m=0}^\infty\binom{j-1+m}{m}\left(-\frac{k}{n}\right)^m\label{bino}
\end{align}
converges only if $k<n$.
Therefore, our approach is to truncate the series in equation (\ref{rf}) and include only those terms where $k\leq \left\lfloor\frac{n}{2}\right\rfloor$:
\begin{align}
\frac{\mathcal{R}_n}{\mathcal{F}_n} = \sum_{k=1}^\infty \frac{ u_{n+k}}{\mathcal{F}_n}
    &= \sum_{k=1}^{\left\lfloor\frac{n}{2}\right\rfloor} \frac{\left(s+\frac{r+sk}{n}\right)t^k}{\left(1+\frac{k}{n}\right)^{3/2}} \exp(\alpha_{n+k)}-\alpha_n)+\mathcal{O}(t^{n/2})\label{neunz}
\end{align}
This introduces an error term of order $\mathcal{O}(t^{n/2})$, which is smaller than $\mathcal{O}(n^{-J})$, since we have assumed $\left|t\right|<1$.

Now we denote
\begin{align}
    D_{n,k}&=\frac{s+\frac{r+sk}{n}}{\left(1+\frac{k}{n}\right)^{3/2}}\exp(\alpha_{n+k}-\alpha_n)\label{Dnk}
\end{align}
and transform $D_{n,k}$ into a function $f$ of two variables:
\begin{align*}
  D_{n,k}&=f(x,y)\quad\text{with } x=\frac{1}{n}\text{ and } y=\frac{k}{n}.
\end{align*}
Using this substitution in (\ref{Dnk}) and $\frac{1}{n+k}=\frac{x}{1+y}$ and Lemma \ref{lem1}, we obtain
\begin{align*}
    f(x,y) &= \frac{s+rx+sy}{(1+y)^{3/2}}\cdot\exp\left(g(x,y)\right)\\
    \text{with }g(x,y) &=\sum_{j=1}^{J-1}f_j\frac{x^j}{(1+y)^j} -\sum_{j=1}^{J-1}f_j x^j +\mathcal{O}\left(\left(\frac{x}{1+y}\right)^{J}\right)
    +\mathcal{O}\left(x^{J}\right)
\end{align*}
Since we have truncated the series to $k\leq\left\lfloor\frac{n}{2}\right\rfloor$, we can choose the following domain of $f$ and $g$:
$$(x,y)\in\left(0;1\right]\times\left(0;\frac{1}{2}\right].$$
Then, $n\rightarrow\infty$ translates into 
$$z=\max\left(\left|x\right|,\left|y\right|\right)\rightarrow 0.$$
Using $\mathcal{O}\left(\frac{x}{1+y}\right) =\mathcal{O}\left(x\right) =\mathcal{O}\left(z\right)$ and $\mathcal{O}\left(x^a y^b\right) =\mathcal{O}\left(z^{a+b}\right)$ and (\ref{bino}), we can find $G_{\mu,\nu}\in\mathbb R$ such that
\begin{align*}
    g(x,y)=\sum_{1\leq\mu+\nu<J}G_{\mu,\nu}x^\mu y^\nu + \mathcal{O}\left(z^J\right)\quad\text{for }z\rightarrow 0.
\end{align*}
Using $\exp(g)=\sum_{l<J}\frac{g^l}{l!}+\mathcal{O}\left(g^J\right)$ and $\mathcal{O}\left(g\right)=\mathcal{O}\left(z\right)$, we can find $H_{\mu,\nu}\in\mathbb R$ such that
\begin{align*}
    h(x,y)=\exp(g(x,y)) &= \sum_{0\leq\mu+\nu<J}H_{\mu,\nu}x^\mu y^\nu + \mathcal{O}\left(z^J\right)\quad\text{for }z\rightarrow 0.
\end{align*}
Using (\ref{bino}) again, we can find $F_{\mu,\nu}\in\mathbb R$ such that
\begin{align*}
    f(x,y)&=\frac{s+rx+sy}{(1+y)^{3/2}}\cdot h(x,y) = \sum_{0\leq\mu+\nu<J}F_{\mu,\nu}x^\mu y^\nu + \mathcal{O}\left(z^J\right)\quad\text{for }z\rightarrow 0.
\end{align*}
Thus we have established the existence of a constant $\Delta\in\mathbb R$ such that the error term $\phi(x,y)$ is uniformly bounded by $\Delta$:
\begin{align}
    f(x,y)&= \sum_{0\leq\mu+\nu<J}F_{\mu,\nu}x^\mu y^\nu + \phi(x,y)\cdot z^J\nonumber\\
    \text{with }\left|\phi(x,y)\right|&<\Delta \text{ for all }(x,y)\in\left(0;1\right]\times\left(0;\frac{1}{2}\right].
\end{align}

Now we translate back into $D_{n,k}=f(x,y)$ with $z=\max\left(\frac{1}{n},\frac{k}{n}\right)=\frac{k}{n}$ and denote $\phi_{n,k}:=\phi\left(\frac{1}{n},\frac{k}{n}\right)$:
\begin{align}
    D_{n,k}&=\sum_{0\leq\mu+\nu<J}F_{\mu,\nu}\frac{k^\nu}{n^{\mu+\nu}} + \phi_{n,k}\cdot\left(\frac{k}{n}\right)^J\label{Dnk2}\\
    \text{with }\left|\phi_{n,k}\right|&<\Delta \text{ for all }n\in\mathbb N\text{ and } k\leq \frac{n}{2}.
\end{align}
Substituting (\ref{Dnk2}) into (\ref{neunz}) results in
\begin{align*}
    \frac{\mathcal{R}_n}{\mathcal{F}_n} 
    &= \sum_{k=1}^{\left\lfloor\frac{n}{2}\right\rfloor}D_{n,k} t^k+\mathcal{O}(t^{n/2})\\
    &= \sum_{k=1}^{\left\lfloor\frac{n}{2}\right\rfloor} \left( \left(\sum_{0\leq\mu+\nu<J}F_{\mu,\nu}\frac{k^\nu}{n^{\mu+\nu}}\right) + \phi_{n,k}\cdot\left(\frac{k}{n}\right)^J \right)\cdot t^k+\mathcal{O}(t^{n/2})\\
    &= \sum_{k=1}^{\left\lfloor\frac{n}{2}\right\rfloor}
         \sum_{j=0}^{J-1}
            \frac{\sum_{\nu=0}^j F_{j-\nu,\nu}\,k^\nu\,t^k}{n^j}
        + {n^{-J}\underbrace{\sum_{k=1}^{\left\lfloor\frac{n}{2}\right\rfloor}\phi_{n,k}\,k^J\,t^k}_{=E_1}} + \mathcal{O}(t^{n/2})\\
        \text{with } |E_1| &\leq \sum_{k=1}^{\left\lfloor\frac{n}{2}\right\rfloor}\Delta\cdot k^J\,t^k \leq \Delta\cdot \sum_{k=1}^{\infty}k^J\,t^k<\infty, \text{ thus } E_1=\mathcal{O}(1).
\end{align*}
Next, we change the summations of $k$ and $j$ and denote $p_j(k)=\sum_{\nu=0}^j F_{j-\nu,\nu}\,k^\nu$, which are polynomials in $k$ of degree $j$:
\begin{align}
    \frac{\mathcal{R}_n}{\mathcal{F}_n} 
    &= \sum_{j=0}^{J-1}
            \frac{\sum_{k=1}^{\left\lfloor\frac{n}{2}\right\rfloor}p_j(k)\,t^k}{n^j}
        +\mathcal{O}(n^{-J}) + \mathcal{O}(t^{n/2}).
\end{align}
Now we re-extend the $k$-summations until $k\rightarrow\infty$, which introduces a third error term:
\begin{align*}
    \frac{\mathcal{R}_n}{\mathcal{F}_n} 
    &= \sum_{j=0}^{J-1}
            \frac{\sum_{k=1}^{\infty}p_j(k)\,t^k+\mathcal{O}\left(\left(\frac{n}{2}\right)^j t^{n/2}\right)}{n^j}
        +\mathcal{O}(n^{-J}) + \mathcal{O}(t^{n/2})\\
    &= \sum_{j=0}^{J-1}
            \frac{\sum_{k=1}^{\infty}p_j(k)\,t^k}{n^j} + \mathcal{O}(t^{n/2})
        +\mathcal{O}(n^{-J}) + \mathcal{O}(t^{n/2}).
\end{align*}
From $\left|t\right|<1$ we obtain $\mathcal{O}(t^{n/2})=\mathcal{O}(n^{-J})$ and
$$c_j=\sum_{k=1}^{\infty}p_j(k)\,t^k<\infty,$$
which finishes the proof of Lemma \ref{lem2}.
\end{proof}


\section{Main result}
\begin{thm}\label{ThmNew} Let $p,q,r,s,t$ be the algebraic coefficients of the series 
\begin{align*}
        \frac{p}{\pi}&= \sum_{k=0}^\infty u_k\qquad\text{with}\qquad u_k=\frac{\left(\frac{1}{2}\right)_k \left(q\right)_k \left(1-q\right)_k}{(k!)^3}(r+s\,k)\,t^k
    \end{align*}
and let $\left|t\right|<1$.
Then the remainder
$$\mathcal{R}_n=\frac{p}{\pi} - \sum_{k=0}^n u_k = \sum_{k=n+1}^\infty u_k$$
admits the following asymptotic expansion with integral $J<\infty$:
\begin{align*}
    \mathcal{R}_n
    = \frac{\left(\frac{1}{2}\right)_n \left(q\right)_n \left(1-q\right)_n}{n!^3}nt^n\left(\sum_{j=0}^{J-1}\frac{c_j}{n^j}+\mathcal{O}\left(n^{-J}\right)\right),\qquad n \rightarrow \infty.
    \end{align*}
To determine the coefficients $(c_j)_{j\in\mathbb N_0}$ in this asymptotic expansion, we first denote
\begin{align*}
d_l &= \begin{cases}
    \hfill t & (l=0)\\
    \hfill-\frac{1}{2}t & (l=1)\\
    \hfill q(1-q)t & (l=2)\\
    \hfill\frac{q(1-q)}{2}t & (l\geq 3)
\end{cases}
\end{align*}
and then apply the following recursive relation:
\begin{align*}
  c_0 &= \frac{d_0 s}{1-d_0},\\
  c_1 &= \frac{d_1 (s+c_0)+d_0 r}{1-d_0},\\
  c_N &= \frac{d_N(s+c_0)+d_{N-1} r + \sum_{k=1}^{N-1} c_k \left(d_{N-k}-\binom{N-1}{k-1}\right)}{1-d_0}\qquad (N\geq 2).
\end{align*}
Caution: The series $\sum_{j=0}^{J-1}\frac{c_j}{n^j}$ diverges if $J\rightarrow\infty$. Therefore, it is necessary to select an appropriate finite $J<\infty$ to approximate $\mathcal{R}_n$.
\end{thm}
\begin{proof}
In Lemma \ref{lem2}, we proved the first part of Theorem \ref{ThmNew}.
This result enables us to express the asymptotic expansion as a formal power series, despite the possible divergence of $\sum_{j=0}^\infty \frac{c_j}{n^j}$:
\begin{align}
        \mathcal{R}_n &\sim \underbrace{\frac{\left(\frac{1}{2}\right)_n \left(q\right)_n \left(1-q\right)_n}{(n!)^3}n\,t^n}_{=:\mathcal{F}_n}\underbrace{\sum_{j=0}^\infty \frac{c_j}{n^j}}_{=:\mathcal{G}_n},\qquad n \rightarrow \infty.\label{asympt2}
\end{align}
Here we defined $\mathcal{F}_n$ and $\mathcal{G}_n$.
Next, the definition of $\mathcal{R}_n$ results in
\begin{align}
    \mathcal{R}_{n-1} = u_n + \mathcal{R}_n.\label{short}
\end{align}
Dividing this by $\mathcal{F}_{n-1}$ gives
\begin{align*}
    \frac{\mathcal{R}_{n-1}}{\mathcal{F}_{n-1}}&= \frac{u_n}{\mathcal{F}_{n-1}} + \frac{\mathcal{R}_{n}}{\mathcal{F}_{n-1}}.
\end{align*}
Using $\mathcal{G}_n$ from (\ref{asympt2}), we can rewrite this equation as follows:
\begin{align}
    \mathcal{G}_{n-1}&\sim \frac{\mathcal{F}_n}{\mathcal{F}_{n-1}}\left(\frac{u_n}{\mathcal{F}_{n}} +\mathcal{G}_n\right).\label{recurr}
\end{align}
But the definition of $\mathcal{F}_n$ yields
\begin{align*}
    \frac{\mathcal{F}_{n}}{\mathcal{F}_{n-1}} 
    &= t\left(1-\frac{1}{2n}\right)\left(1-\frac{q}{n}\right)\left(1-\frac{1-q}{n}\right) \frac{n}{n-1}\nonumber\\
    &= t\left(1-\frac{3}{2n}+\frac{1+2q(1-q)}{2n^2}-\frac{q(1-q)}{2n^3}\right) \sum_{k=0}^\infty \frac{1}{n^k}\\
    &= t-\frac{t}{2n}+\frac{q(1-q)t}{n^2}+\sum_{l=3}^\infty \frac{q(1-q)t}{2n^l}
    =: \sum_{l=0}^\infty \frac{d_l}{n^l}.
\end{align*}
This explains the values of $(d_l)_{l\in\mathbb N_0}$ from Theorem \ref{ThmNew}.

Since we have $\frac{u_n}{\mathcal{F}_{n}} = s + \frac{r}{n}$, equation (\ref{recurr}) leads to
\begin{align*}
 \underbrace{\sum_{j=0}^\infty \frac{c_j}{(n-1)^j}}_{\mathcal{G}_{n-1}} &\sim \underbrace{\sum_{l=0}^\infty \frac{d_l}{n^l}}_{\frac{\mathcal{F}_n}{\mathcal{F}_{n-1}}}\Biggl(\underbrace{\vphantom{\sum_{j=0}^\infty}s + \frac{r}{n}}_{\frac{u_n}{\mathcal{F}_{n}}} +\underbrace{\sum_{j=0}^\infty \frac{c_j}{n^j}}_{\mathcal{G}_n}\Biggr).
\end{align*}
Using $\frac{1}{\left(1-x\right)^j}
    = \sum_{l=0}^\infty\binom{l+j-1}{j-1} x^l$ with $x=\frac{1}{n}$, we obtain
\begin{align*}
     \mathcal{G}_{n-1} \sim c_0 + \sum_{j=1}^\infty \frac{c_j}{(n-1)^j}
     \sim c_0+\sum_{j=1}^\infty\sum_{l=0}^\infty \frac{\binom{l+j-1}{j-1}c_j}{n^{j+l}}.
\end{align*}
Then we have
\begin{align}
          c_0+\sum_{j=1}^\infty\sum_{l=0}^\infty \frac{\binom{l+j-1}{j-1}c_j}{n^{j+l}}
    &\sim \sum_{l=0}^\infty \frac{d_l}{n^l}\left(s + \frac{r}{n} +\sum_{j=0}^\infty \frac{c_j}{n^j}\right)\label{main}.
\end{align}
To prove the recursive relation between the coefficients $c_j$, we start by equating the constant terms in equation (\ref{main}). This gives us $c_0 = d_0(s+c_0)$, which can be solved for $c_0$ as follows:
\begin{align*}
    c_0 = \frac{d_0 s}{1-d_0}.
\end{align*}
Next, equating the linear terms yields $c_1 = d_1(s+c_0)+d_0(r+c_1)$ and thus
\begin{align*}
    c_1 = \frac{d_1(s+c_0)+d_0 r}{1-d_0}.
\end{align*}
For all higher orders where $N=j+l\geq 2$, we obtain the relation
\begin{align*}
    c_N =\frac{d_N(s+c_0)+d_{N-1} r + \sum_{j=1}^{N-1} c_j \left(d_{N-j}-\binom{N-1}{j-1}\right)}{1-d_0}.
\end{align*}
With this, the proof of our main result (Theorem \ref{ThmNew}) is complete.
\end{proof}

In our proof, we required $\left|t\right|<1$. Consequently, we have not proven the existence of an asymptotic expansion for series \#33 in Table \ref{L0N1}. However, we can compute the values of $c_j$ using the recursive formula from Theorem \ref{ThmNew} and verify numerically if the results really approximate $\mathcal{R}_n$ for this series. We observed that the signs of the $c_j$ in series \#33 appear to follow a pattern:
\begin{align*}
\sum_{j}\frac{c_j}{n^j} = -2 - \frac{1}{4n^3} + \frac{1}{8n^4} + \frac{3}{8n^5} - \frac{13}{32n^6} - \frac{83}{64n^7} + \frac{141}{64n^8} + \frac{2081}{256n^9} - - + + \ldots
\end{align*}
This observation led us to propose the following conjecture:
\begin{conj}
    Theorem \ref{ThmNew} is also valid for series \#33 in Table \ref{L0N1}, which has $t=-1$ and not $\left|t\right|<1$.
    For this series, we even have an enveloping expansion:
\begin{align}
    \sum_{j=0}^{4L-1}\frac{c_j}{n^j} < \frac{\mathcal{R}_n}{\mathcal{F}_n} < \sum_{j=0}^{4L+1}\frac{c_j}{n^j}\qquad \text{for }L\in\mathbb N.\label{conject}
\end{align}
\end{conj}
\noindent We verified (\ref{conject}) numerically for all $L\leq 1000$ and for all $n \leq 1000$.

\section{Ramanujan's and the Chudnovskys' series}
In this section, we apply Theorem \ref{ThmNew} to Ramanujan's fastest series (\#23 in Table \ref{L0N1}) and to the Chudnovskys' series (\#7 in Table \ref{L0N1}) and give the first few coefficients of their asymptotic expansions:
\begin{cor}
    The remainder $\mathcal{R}_n$ of Ramanujan's fastest series
    \begin{align*}
        \mathcal{R}_n=\frac{9801\sqrt{2}}{4\pi} - \sum_{k=0}^{n}\frac{1}{2^{8k}}\frac{(4k)!}{(k!)^4}\,\frac{1103+26390\,k}{99^{4k}}
    \end{align*}
     admits the following asymptotic expansion for $n\rightarrow\infty$:
     \begin{dmath*}
        \mathcal{R}_n \sim \frac{1}{2^{8n}}\frac{(4n)!}{(n!)^4}\frac{n}{99^{4n}}
        \left(\frac{1}{3640}
  +\frac{-3035509}{24114272000n}
  +\frac{27421461880263}{159752229145600000n^2}
  +\frac{-40112960459081444847}{211665313528754176000000n^3}
  +\frac{41325245596206392083139943}{200320052723612952166400000000n^4}
  +\frac{-294617758251626753628627256762503}{1327080285283391085511966720000000000n^5}
+\ldots\right).
     \end{dmath*}
     These are the coefficients from Han's and Chen's Theorem 2.1 \cite{HanChen}.
\end{cor}

\begin{cor}
    The remainder $\mathcal{R}_n$ of the Chudnovskys' series
    \begin{align*}
        \mathcal{R}_n=\frac{426880\sqrt{10005}}{\pi} - \sum_{k=0}^n \frac{(6k)!}{12^{3k}(3k)!(k!)^3} \frac{13591409 + 545140134\,k}{(-53360)^{3k}}
    \end{align*}
     admits the following asymptotic expansion for $n\rightarrow\infty$:
     \begin{dmath*}
        {\mathcal{R}_n \sim \frac{(6n)!}{12^{3n}(3n)!(n!)^3}\frac{n}{(-53360)^{3n}}}
        \left(\frac{-2}{557403}+\frac{885616447271}{519552166475669481n}
  +\frac{-348383485711899665152000}{161423873265579705965912841729n^2}
  + \frac{353617726049706364359910891893760000}{150462274289922801810326124883532202249483n^3}
+ \frac{-50809055546960209063426859112570403004227648000}{20035004317262584720034794527173169059667245914401663n^4}\nolinebreak
+\ldots\right).
     \end{dmath*}
     These are the coefficients from Han's and Chen's open problem 3.1 \cite{HanChen}.
\end{cor}

\end{document}